\documentclass[11pt,reqno]{amsart}
\usepackage{graphicx}

\usepackage{amssymb}
\usepackage{cite}
\usepackage{amsmath}
\usepackage{latexsym}
\usepackage{amscd}
\usepackage{amsthm}
\usepackage{mathrsfs}
\usepackage{url}

\usepackage[backref=page]{hyperref}
\numberwithin{equation}{section}

\newtheorem{thm}{Theorem}

\newtheorem{prop}[thm]{Proposition}

\newtheorem{exam}{Example}[section]

\theoremstyle{definition}
\newtheorem{defn}{Definition}[section]
\theoremstyle{remark}
\newtheorem{rem}[thm]{Remark}
\newtheorem*{ack}{Acknowledgment}
\def\C{\mathbb C}
\def\R{\mathbb R}

\def\SS{\mathbb S}
\def\cal{\mathcal}

\def\pt{\partial}

\allowdisplaybreaks

\begin{document}
\title[Weighted Reilly formula and Steklov eigenvalue estimates]{A weighted Reilly formula for differential forms and sharp Steklov eigenvalue estimates}
\author[C.~Xiong]{Changwei~Xiong}
\address{School of Mathematics, Sichuan University, Chengdu 610065, Sichuan,  P.~R.~China}
\email{\href{mailto:changwei.xiong@scu.edu.cn}{changwei.xiong@scu.edu.cn}}
\date{\today}
\subjclass[2010]{{35P15}, {47A75}, {49R05}, {35P20}}
\keywords{Reilly formula; differential form; Steklov eigenvalue problem; eigenvalue estimate; weight function}

\maketitle

\begin{abstract}
First we establish a weighted Reilly formula for differential forms on a smooth compact oriented Riemannian manifold with boundary. Then we give two applications of this formula when the manifold satisfies certain geometric conditions. One is a sharp lower bound for the first positive eigenvalue of the Steklov eigenvalue problem on differential forms investigated by Belishev and Sharafutdinov (2008) and Karpukhin (2019). The other one is a comparison result between the spectrum of this Steklov eigenvalue problem and the spectrum of the Hodge Laplacian on the boundary of the manifold. Besides, at the end we discuss an open problem for differential forms analogous to Escobar's conjecture (1999) for functions.
\end{abstract}

\section{Introduction}\label{sec1}

In the differential geometry and geometric analysis, the classical Reilly formula \cite{Rei77} has been a crucial and extensively-studied tool and has led to lots of interesting applications. See e.g. \cite{CW83,PRS03} and references therein. Moreover, this formula admits various extensions, which produces further applications. The extensions include the weighted version \cite{QX15} for functions, the version \cite{RS11} for differential forms, the versions \cite{HMZ01,CJSZ19} in the spin geometry, the version \cite{GKY13} on K\"{a}hler manifolds, and the versions \cite{KM17,MD10} on Riemannian manifolds with density. In this paper we prove the following weighted Reilly formula for differential forms.
\begin{thm}\label{thm-Reilly}
Let $(M^{n+1},g)$ be an $(n+1)$-dimensional smooth compact connected oriented Riemannian manifold with boundary $\Sigma=\pt M$. Let $N$ be the inner unit normal vector field along $\Sigma$ and $J:\Sigma \rightarrow M$ the inclusion mapping. Let $\omega\in \Omega^p(M)$ be a differential form of the degree $p$ and $f\in C^\infty(M)$ a smooth function on $M$. There holds
\begin{align*}
&\int_M   f\left( |\delta \omega |^2+ |d\omega |^2- |\nabla \omega |^2\right)dv \\
=&-\int_M  2\langle  \omega,  i_{\nabla f}(d\omega) \rangle dv +\int_M \langle \omega, \nabla^2 f(\omega)\rangle+\Delta f\cdot |\omega |^2+f \langle W_{M}^{[p]}(\omega),\omega\rangle dv \\
&-\int_\Sigma f_N  |J^*\omega |^2  da+2\int_\Sigma f\langle \delta^\Sigma (J^*\omega),i_N\omega \rangle da+\int_\Sigma f\mathcal{B}(\omega,\omega) da,
\end{align*}
where $i$ denotes the interior product, $\nabla^2 f (\omega)$ the operation on $\Omega^p(M)$ induced by $\nabla^2 f$ as a $(1,1)$ tensor (see Definition~\ref{def-2.1} and Example~\ref{exa-2.2}), $W_M^{[p]}$ the Weitzenb\"{o}ck curvature, the term $\mathcal{B}(\omega,\omega)$ is given by
\begin{equation}\label{eq-B}
\mathcal{B}(\omega,\omega)=\langle  S^{[p]}(J^*\omega), J^*\omega\rangle+nH|  i_N\omega |^2-\langle S^{[p-1]}(i_N\omega) ,i_N \omega\rangle,
\end{equation}
the function $H$ is the mean curvature of the boundary $\Sigma$, and $S^{[k]}$ is the operator on $\Omega^{k}(\Sigma)$ induced by the shape operator $S:T\Sigma\to T\Sigma$ (see Definition~\ref{def-2.1} and Example~\ref{exa-2.1}).
\end{thm}

\begin{rem}
When taking $\omega=du$ for a smooth function $u\in C^\infty(M)$, we recover a special case in \cite{QX15}. Namely, for two smooth functions $f$ and $u$ on $M$, there holds
\begin{align}\label{xeq0}
&\int_M f\left((\Delta u)^2-|\nabla^2 u|^2\right)dv\nonumber\\
=&\int_\Sigma f\left(2 u_N\Delta_\Sigma u+nH(u_N)^2+h(\nabla_\Sigma u, \nabla_\Sigma u)\right)da\nonumber \\
&{}-\int_\Sigma   f_N\, |\nabla_\Sigma u|^2  da +\int_M \left(\nabla^2 f +\Delta f   g+ f \mathrm{Ric}_g\right)(\nabla u, \nabla u)dv.
\end{align}
\end{rem}

\begin{rem}
When $f=1$, we recover the formula from \cite{RS11} (also cf. Theorems~2.1.5 and 2.1.7 in \cite{Sch95}). That is, for a differential form $\omega\in \Omega^p(M)$ of the degree $p$
there holds
\begin{align}\label{eq-RS}
&\int_M  |d\omega |^2+ |\delta \omega |^2 dv=\int_M  |\nabla \omega  |^2+\langle W_{M}^{[p]}(\omega),\omega\rangle dv\nonumber \\
&+2\int_\Sigma \langle \delta^\Sigma (J^*\omega),i_N \omega \rangle da+\int_\Sigma \mathcal{B}(\omega,\omega) da,
\end{align}
where $\mathcal{B}(\omega,\omega)$ is given in \eqref{eq-B}.
\end{rem}
The proof of Theorem~\ref{thm-Reilly} starts with the standard Bochner formula for differential forms. Then we integrate an identity resulting from the Bochner formula and employ multiple times the Stokes formula. After suitably arranging the terms both in the interior integral and in the boundary integral, we arrive at the demanded formula.

Next we discuss the applications of the weighted Reilly formula. As the first application, we get a sharp lower bound for the first non-zero eigenvalue of a Steklov eigenvalue problem for differential forms. In the literature, roughly speaking, there are three types of the Steklov eigenvalue problems (or the Dirichlet-to-Neumann maps) for differential forms. One is introduced and studied by Belishev and Sharafutdinov \cite{BS08} and Karpukhin \cite{Kar19}. A second one, the complete Dirichlet-to-Neumann map, is by Joshi and Lionheart \cite{JL05} and Sharafutdinov and Shonkwiler \cite{SS13}. (Indeed, this one determines the previous one.) Another one is by Raulot and Savo \cite{RS12} (cf. Carron's relative to absolute operator \cite{Car02}). In this paper we are mainly concerned with the one in \cite{BS08,Kar19}. Precisely, the problem may be described as follows. Given an integer $0\leq p\leq n$ and $\phi \in \Omega^p(\Sigma)$, there exists a solution $\varphi\in \Omega^p(M)$ satisfying
\begin{align}\label{eq-DN}
\begin{cases}
\Delta \varphi=0,\ \delta \varphi  =0\text{ on }M,\\
J^*\varphi =\phi\text{ on }\Sigma.
\end{cases}
\end{align}
Then the Dirichlet-to-Neumann operator $\mathcal{D}^{[p]}$ is defined by
\begin{equation*}
\mathcal{D}^{[p]}(\phi):=-i_Nd\varphi\in \Omega^p(\Sigma).
\end{equation*}
In order to study the spectral properties of $\mathcal{D}^{[p]}$, it is advisable to restrict $\mathcal{D}^{[p]}$ to the subspace of co-closed $p$-forms in $\Omega^p(\Sigma)$ (see \cite{Kar19}). In that case the restricted operator admits a sequence of eigenvalues
\begin{align*}
0\leq \sigma_1^{[p]}(\Sigma)\leq \sigma_2^{[p]}(\Sigma)\leq \dots \nearrow \infty.
\end{align*}

When $p=0$, the above problem reduces to the classical Steklov eigenvalue problem for functions. This classical problem was introduced by W.~Steklov around 1900 (\cite{Ste02,KKK14}) and since then has been attracting considerable attention from mathematicians. For an account of its developments, we refer to the nice surveys \cite{GP17,CGGS23}. In addition, for $p=0$ the results similar to those in the following Theorems~\ref{thm-bound} and \ref{thm-comparison} have been obtained in \cite{XX23}. So in the applications below we only consider the case $1\leq p\leq n$. We prove the following result.
\begin{thm}\label{thm-bound}
Let $(M^{n+1},g)$ be an $(n+1)$-dimensional smooth compact connected oriented Riemannian manifold with boundary $\Sigma=\pt M$ and $1\leq p\leq n$ an integer. Assume that the $(p+1)$st Weitzenb\"{o}ck curvature and the sectional curvature of $M$ are non-negative, and the principal curvatures of $\Sigma$ in $M$ are no less than $c>0$. Denote by $\sigma$ the first non-zero eigenvalue of the Dirichlet-to-Neumann operator $\mathcal{D}^{[p]}$. Then
\begin{equation*}
\sigma\geq (p+1)c.
\end{equation*}
The equality holds for a Euclidean ball with the radius $1/c$. Conversely, if the equality holds, then there exists a non-trivial co-closed $p$-form $\varphi\in \Omega^p(M)$ (``non-trivial'' means that $\varphi$ is not a harmonic field with the Dirichlet boundary condition) such that
\begin{align}\label{eq-rigidity}
\begin{cases}
\nabla d\varphi=0\text{ on }M,\\
-i_Nd\varphi=(p+1)c J^*\varphi\text{ on }\Sigma.
\end{cases}
\end{align}
\end{thm}
\begin{rem}
  For $p=0$, the analogous result as in Theorem~\ref{thm-bound} is related to the well-known Escobar's conjecture \cite[p.~115]{Esc99}. The Escobar's conjecture is stated for Riemannian manifolds with non-negative Ricci curvature; the author joint with Chao~Xia \cite{XX23} has confirmed this conjecture for Riemannian manifolds with non-negative sectional curvature. See \cite{Esc97,Esc99,XX23} for more details.
\end{rem}
\begin{rem}
  For the Dirichlet-to-Neumann operator $\mathcal{T}^{[p]}$ investigated by Raulot and Savo (see Section~\ref{sec2.2} for some introductions), they have obtained lower bounds for the first non-zero eigenvalue $\nu_1^{[p]}$ (\cite[Thm.~1]{RS12}), which are sharp in the range $(n+1)/2\leq p\leq n$. On the one hand, we note that by Karpukhin's \cite[Thm.~2.5]{Kar19} there holds  $\nu_1^{[p]}\leq \sigma^{[p]}_1$ for $0\leq p\leq n-1$ (in our Theorem~\ref{thm-bound} $\sigma=\sigma_1^{[p]}$). On the other hand, we emphasize that in Theorem~\ref{thm-bound} we assume $W_{M}^{[p+1]}\geq 0$, while in Raulot and Savo's \cite[Thm.~1]{RS12}, they assume $W_{M}^{[p]}\geq 0$.
\end{rem}
\begin{rem}
  Using the Hodge star operator $*$, we see $*W_M^{[p]}=W_M^{[n+1-p]}*$ on $\Omega^p(M)$, and so $W_M^{[p]}\geq 0$ if and only if $W_M^{[n+1-p]}\geq 0$. We also know that $W_M^{[1]}=\mathrm{Ric}$ and that $W_M^{[2]}\geq 0$ implies $\mathrm{Ric}\geq 0$ (see e.g. Sec.~2 in \cite{Kar17} or Sec.~9.4.5 in \cite{Pet16}). By virtue of these facts, the assumption $W_{M}^{[p+1]}\geq 0$ in Theorem~\ref{thm-bound} when $p=n-1,n$ is redundant, since the sectional curvature of $M$ is assumed to be non-negative.
\end{rem}
\begin{rem}
It would be of interest to see whether under the conditions of Theorem~\ref{thm-bound} the equation \eqref{eq-rigidity} implies that the manifold $M$ is isometric to a Euclidean ball with the radius $1/c$. This is an Obata-type problem (\cite{Oba62}). Currently we do not know the answer for the general case. The corollaries~10.110 and 10.111 in \cite{Bes08} and the paper \cite{Ver11} might provide some hints. Nevertheless, the case $p=n$ is quite special; in this case one can get the rigidity result as in \cite[Thm.~2]{RS12}.
\end{rem}
\begin{rem}
See the last Section~\ref{sec6} for a non-sharp lower bound of $\sigma$ and a related open problem.
\end{rem}
For the proof of Theorem~\ref{thm-bound}, we mainly utilize two ingredients, the weighted Reilly formula in Theorem~\ref{thm-Reilly} and a stimulating Pohozhaev-type identity (Proposition~\ref{prop-Pohozhaev} below) from \cite{GKLP22}. In \cite{GKLP22}, among others, A.~Girouard, M.~Karpukhin, M.~Levitin, and I.~Polterovich exploit the use of Pohozhaev-type identities on eigenvalue estimates and spectral asymptotics. This kind of applications of Pohozhaev-type identities, as also in \cite{CGH20,HS20,PS19,Xio18}, proves very powerful and may date back to Pohozhaev's paper \cite{Poh65}, or even further to a rediscovered manuscript \cite{Hor18} written by L.~H\"{o}rmander in 1950s. Besides, in the proof of Theorem~\ref{thm-bound} another key point worth mentioning is to choose a suitable weight function in the weighted Reilly formula and a suitable vector field in the Pohozhaev-type identity, the idea of which comes from the author's joint work \cite{XX23} with Chao~Xia.

As the second application, we obtain a comparison result between the spectrum of the Dirichlet-to-Neumann operator $\mathcal{D}^{[p]}$ and the spectrum of the Hodge Laplacian $\Delta_\Sigma$ on the boundary $\Sigma$ restricted to the co-closed differential forms.
\begin{thm}\label{thm-comparison}
Assumptions are the same as in Theorem~\ref{thm-bound} except we suppose $1\leq p\leq n-1$. Denote by $\sigma_k^{[p]}$ ($k\geq 1$) the eigenvalues of the operator $\mathcal{D}^{[p]}$ and by $\lambda_k^{[p]}$ ($k\geq 1$) the eigenvalues of the Hodge Laplacian $\Delta_\Sigma$, both operators being restricted to the co-closed differential forms. Then
\begin{align*}
\sigma^{[p]}_k&\leq \frac{1}{(n-p)c} \lambda_k^{[p]},\quad k\geq 1.
\end{align*}
When $M$ is a Euclidean ball with the radius $1/c$, the equality is attained for $1\leq k\leq C_{n+1}^{p+1}$.
\end{thm}
\begin{rem}
  In the literature, there are other types of interesting results comparing the Steklov eigenvalues on the manifold $M$ and the Laplacian eigenvalues on the boundary $\Sigma$. See e.g. \cite{CGH20,Esc99,GKLP22,Kar17,Kar19,Kwo16,RS12,PS19,WX09,Xio18,XX23,YY17}.
\end{rem}
For the proof of Theorem~\ref{thm-comparison}, we mainly use the weighted Reilly formula in Theorem~\ref{thm-Reilly} with a suitably chosen weight function as mentioned before, and the variational characterizations for $\sigma_k^{[p]}$ and $\lambda_k^{[p]}$.

The paper is structured as follows. In Section~\ref{sec2} we present some preliminaries on the differential forms, the involved eigenvalue problems, a Pohozhaev-type identity, and a weight function together with its approximation. In Section~\ref{sec3} we prove the weighted Reilly formula for differential forms. In Sections~\ref{sec4} and \ref{sec5} we give the proofs of Theorems~\ref{thm-bound} and \ref{thm-comparison} respectively. In the last section we discuss a non-sharp lower bound for the first non-zero eigenvalue of $\mathcal{D}^{[p]}$ and a related problem.

\begin{ack}
The author wishes to thank Ben Andrews, Alexandre Girouard, Han Hong, Michael Levitin, Haizhong Li, Martin Li, Iosif Polterovich, Simon Raulot for their interest. In particular, he is grateful to Han Hong for a question which results in the appearance of Section~\ref{sec6}. This research was supported by National Key R and D Program of China 2021YFA1001800 and NSFC (Grant no. 12171334).
\end{ack}

\section{Preliminaries}\label{sec2}

In this section we collect some preparatory materials which are needed in the proofs of our results, including differential forms, relevant eigenvalue problems on differential forms, a Pohozhaev-type identity in the setting of differential forms, and the construction of a specific weight function.

\subsection{Differential forms}

Let $(M^{n+1},g)$ be an $(n+1)$-dimensional smooth compact connected oriented Riemannian manifold with boundary $\Sigma=\pt M$. Let $N$ be the inner unit normal vector field along $\Sigma$. Let $J:\Sigma \rightarrow M$ be the inclusion mapping. Let $\Omega^p(M)$ be the space of sections of the vector bundle $\Lambda^p(T^*M)$ of smooth exterior differential forms of the degree $p$. Let $d$ and $\delta$ be the differential and co-differential operators on $\Omega^p(M)$, respectively. Let $\mathcal{H}^p(M)$, $\mathcal{E}^p(M)$, $c\mathcal{E}^p(M)$, $\mathcal{C}^p(M)$, and $c\mathcal{C}^p(M)$ denote respectively the subspace in $\Omega^p(M)$ of harmonic fields (i.e., closed and co-closed forms), exact forms, co-exact forms, closed forms, and co-closed forms. Similarly we can define spaces and operators on the boundary $\Sigma=\pt M$, e.g., $\Omega^p(\Sigma)$, $d^\Sigma$, $\delta^\Sigma$. We use $\langle \cdot,\cdot \rangle$ and $|\cdot|$ to designate the pointwise inner product and the norm respectively on the differential forms induced by the Riemannian metric $g$, and use $dv$ and $da$ to denote the volume element and area element on $M$ and $\Sigma$, respectively. First we recall the Stokes formula, which will be used frequently later.
\begin{prop}
For any two differential forms $\varphi, \psi \in \Omega^p(M)$, there holds
\begin{equation}
\int_M \langle d\varphi, \psi\rangle dv =\int_M \langle \varphi, \delta \psi\rangle dv -\int_{\Sigma} \langle J^* \varphi, i_N\psi\rangle da.
\end{equation}
\end{prop}
\begin{rem}
Here $i_N\psi$ is viewed as an element in $\Omega^{p-1}(\Sigma)$.
\end{rem}
Recall that the Hodge Laplacian on differential forms is given by
\begin{equation}
\Delta \omega =d\delta \omega+\delta d\omega,\ \omega \in \Omega^p(M).
\end{equation}
Then we have the classical Bochner formula
\begin{equation}
\Delta \omega =\nabla^*\nabla \omega+W_{M}^{[p]}(\omega),
\end{equation}
where $\nabla^*\nabla$ denotes the connection Laplacian on differential forms and $W_{M}^{[p]}:\Omega^p(M)\to \Omega^p(M)$ is the Weitzenb\"{o}ck curvature operator given by (see e.g. \cite[Thm.~9.4.1]{Pet16})
\begin{align*}
&W_M^{[p]}(\omega)(X_1,X_2,\dots,X_p)\\
=\sum_{i,j}(Rm(e_i,X_j)&(\omega))(X_1,X_2,\dots,X_{j-1},e_i,X_{j+1},\dots, X_p),\ \{X_k\}_{k=1}^p\subset \mathfrak{X}(M).
\end{align*}
Here $Rm(X,Y)=\nabla_X\nabla_Y-\nabla_Y\nabla_X-\nabla_{[X,Y]}:\Omega^p(M)\to \Omega^p(M)$ denotes the induced Riemann curvature operator acting on the differential forms and $\{e_i\}_{i=1}^{n+1}$ is a local orthonormal frame for the space $\mathfrak{X}(M)$ of smooth vector fields. By Gallot and Meyer's work \cite{GM75}, if the eigenvalues of the curvature operator are bounded below by $\gamma\in \R$, then
\begin{equation*}
\langle W_M^{[p]}(\omega),\omega\rangle \geq p(n+1-p)\gamma |\omega|^2,\quad \omega \in \Omega^p(M).
\end{equation*}
\begin{rem}
In our convention note that both $\Delta$ and $\nabla^*\nabla$ are non-negative operators. Thus, for a smooth function $f$ and an orthonormal frame $\{e_i\}_{i=1}^{n+1}$, we get
\begin{equation}
\Delta f=-\sum_{i=1}^{n+1} f_{ii}.
\end{equation}
\end{rem}

On the boundary $\Sigma=\pt M$, we use $S:\mathfrak{X}(\Sigma)\to \mathfrak{X}(\Sigma)$ to denote the Weingarten operator (shape operator) with respect to $N$, i.e., $S(X)=-\nabla_X N$ for $X\in \mathfrak{X}(\Sigma)$. Then $H=\mathrm{tr}_g S/n$ is the mean curvature of the boundary.

Next we define the operation on differential forms for a $(1,1)$ tensor as follows.
\begin{defn}\label{def-2.1}
Given a $(1,1)$ tensor $T$ on a Riemannian manifold $(\widetilde{M}^m,g)$, it induces an operator $T^{[p]}:\Omega^p(\widetilde{M})\to \Omega^p(\widetilde{M})$, $0\leq p\leq m$, by
\begin{equation*}
(T^{[p]}(\omega))(X_1,X_2,\dots,X_p):=\sum_{i=1}^p\omega(X_1,X_2,\dots,T(X_i),\dots,X_p),\quad \omega\in \Omega^p(\widetilde{M}).
\end{equation*}
By convention, we set $T^{[0]}=0$.
\end{defn}
Here we discuss two examples.
\begin{exam}\label{exa-2.1}
The Weingarten operator $S$ induces an operator $S^{[p]}:\Omega^p(\Sigma)\to \Omega^p(\Sigma)$, $0\leq p\leq n$. The eigenvalues of $S^{[p]}$ are called the $p$-curvatures of the boundary $\Sigma$. In other words, let $\{\kappa_i\}_{i=1}^n$ be the principal curvatures of $S$ in the non-decreasing order. Then for any $1\leq i_1<i_2<\cdots<i_p\leq n$, the sum
\begin{equation*}
\kappa_{i_1}+\kappa_{i_2}+\cdots+\kappa_{i_p}
\end{equation*}
is a $p$-curvature of $\Sigma$. Define the lowest $p$-curvature as
\begin{equation*}
\sigma_p(x)=\kappa_{1}+\kappa_{2}+\cdots+\kappa_{p},
\end{equation*}
and set $\sigma_p(\Sigma)=\inf_{x\in \Sigma}\sigma_p(x)$. We say that $\Sigma$ is $p$-convex if $\sigma(\Sigma)\geq 0$. In particular, the $1$-convexity is the usual convexity and the $n$-convexity is the mean convexity. Moreover, we see $\sigma_q/q\geq \sigma_p/p$ if $q\geq p$.
\end{exam}

\begin{exam}\label{exa-2.2}
  Given a Lipschitz vector field $F$ on $M$, then $\nabla F$ is a $(1,1)$ tensor and it induces an operator $\nabla F^{[p]}:\Omega^p(M)\to \Omega^p(M)$ on the $p$-differential forms. In particular, if $F=\nabla f$ is given by the gradient vector field of a smooth function $f\in C^\infty(M)$, then we may view its Hessian $(1,1)$ tensor $\nabla^2 f^{[p]}=\nabla (\nabla f)^{[p]}$ as an operator on $\Omega^p(M)$. We will use $\nabla F$ for $\nabla F^{[p]}$ when the context is clear.
\end{exam}

\subsection{Eigenvalue problems}\label{sec2.2}
In this subsection we discuss three types of eigenvalue problems for differential forms. First we consider the following Steklov eigenvalue problem for differential forms, which was introduced in \cite{BS08, Kar19}. This eigenvalue problem has some motivation relating to the investigations of the Maxwell's equations in mathematical physics. Given $\phi \in \Omega^p(\Sigma)$, there exists a solution $\varphi\in \Omega^p(M)$ satisfying (see \cite[Lemma~3.4.7]{Sch95})
\begin{align}\label{eq-DN}
\begin{cases}
\Delta \varphi=0,\ \delta \varphi  =0\text{ on }M,\\
J^*\varphi =\phi\text{ on }\Sigma;
\end{cases}
\end{align}
the solution is unique up to an element of $\mathcal{H}_D^p(M)$, the space of harmonic fields with the Dirichlet boundary condition. For definiteness, we may fix the solution $\varphi$ such that it is orthogonal to $\mathcal{H}_D^p(M)$ (see Proposition~3.11 in \cite{Kar19}). Then the Dirichlet-to-Neumann operator $\mathcal{D}^{[p]}$ is defined by
\begin{equation*}
\mathcal{D}^{[p]}(\phi):=-i_Nd\varphi\in \Omega^p(\Sigma).
\end{equation*}
By \cite{Kar19}, the followings hold true.
\begin{thm}
The operator $\mathcal{D}^{[p]}:\Omega^p(\Sigma)\to \Omega^p(\Sigma)$ is well-defined and self-adjoint. Moreover, it admits the following properties.
\begin{enumerate}
  \item There hold $\mathrm{ker}(\mathcal{D}^{[p]})=J^*\mathcal{H}^p(M)$ and $\mathcal{E}^p(\Sigma)\subset \mathrm{ker}( \mathcal{D}^{[p]})$.
  \item The restriction on the space of co-closed forms $\mathcal{D}^{[p]}:c\mathcal{C}^p(\Sigma)\to c\mathcal{C}^p(\Sigma)$ is an operator with a compact resolvent. The eigenvalues of the restriction consist of a sequence
      \begin{align*}
0\leq \sigma_1^{[p]}(\Sigma)\leq \sigma_2^{[p]}(\Sigma)\leq \dots \nearrow \infty,
\end{align*}
with the account of multiplicities.
  \item The eigenvalues satisfy the variational principle
  \begin{align}
\sigma_k^{[p]}(\Sigma)=\inf_{E_k }\sup_{\varphi \in E_k\setminus \{0\}}\frac{\int_M  |d  \varphi |^2 dv}{\int_\Sigma |J^*\varphi |^2 da},
\end{align}
where $E_k$ ranges over $k$-dimensional subspaces in $\Omega^p(M)$ satisfying $J^*E_k\subset c\mathcal{C}^p(\Sigma)$.
\end{enumerate}
\end{thm}
To check the sharpness of our results, we present the spectrum of $\mathcal{D}^{[p]}$ for the unit Euclidean ball $B^{n+1}(1)$.
\begin{exam}[{\cite{GM75},\cite{IT78},\cite[Prop.~7]{RS14},\cite[Thm.~8.1]{Kar19}}]\label{exam1}
  Let $M=\overline{B^{n+1}(1)}$ be the unit Euclidean ball. Let $P_{l,p}$ be the space of homogeneous polynomial $p$-forms of the degree $l\geq 0$ in $\R^{n+1}$. Define three subspaces of $P_{l,p}$ as follows:
  \begin{align*}
  H_{l,p}&=\{\omega\in P_{l,p}|\Delta \omega =0,\ \delta \omega=0\},\\
  H'_{l,p}&=\{\omega\in H_{l,p}|d\omega=0\},\\
  H''_{l,p}&=\{\omega \in H_{l,p}|i_N\omega=0\text{ on }\SS^n\}.
  \end{align*}
  (i) Let $1\leq p\leq n-1$. Then $\mathcal{H}^p(\SS^n)=0$ and $\Omega^p(\SS^n)=\mathcal{E}^p(\SS^n)\oplus c\mathcal{E}^p(\SS^n)$. By \cite{IT78}, we know that $\mathcal{E}^p(\SS^n)=\bigoplus_l(J^*H'_{l,p})$, $c\mathcal{E}^p(\SS^n)=\bigoplus_l(J^*H''_{l,p})$, and
  \begin{equation*}
  \delta^{\SS^n}:J^*H'_{l,p}\to J^*H''_{l+1,p-1}
  \end{equation*}
  is an isomorphism. So $\mathrm{dim} J^* H''_{1,p}=\mathrm{dim} J^* H'_{0,p+1}=C_{n+1}^{p+1}$. By Theorem~8.1 in \cite{Kar19}, the spaces $J^* H'_{l-1,p}$ and $J^*H''_{l,p}$ for $l\geq 1$ form the eigenbases of $\mathcal{D}^{[p]}$. Precisely, if $\phi\in J^* H'_{l-1,p}$, then $\mathcal{D}^{[p]}\phi=0$; if $\phi\in J^* H''_{l,p}$, then $\mathcal{D}^{[p]}\phi=(p+l)\phi$. In particular, the first positive eigenvalue of $\mathcal{D}^{[p]}$ is $p+1$ with multiplicity $C_{n+1}^{p+1}$. (ii) Let $p=n$. If $\phi\in J^* H'_{l-1,n}$, then $\mathcal{D}^{[n]}\phi=0$; if $\phi$ is the volume form of $\SS^n$, then $\mathcal{D}^{[n]}\phi=(n+1)\phi$.
\end{exam}

Second, for the purpose of comparison, we introduce the Steklov eigenvalue problem investigated by Raulot and Savo in \cite{RS12}. Given a $p$-form $\phi\in \Omega^p(\Sigma)$ of the degree $p=0,1,\dots,n$, there exists a unique $p$-form $\varphi\in \Omega^p(M)$ satisfying (see \cite{Sch95})
\begin{align}\label{eq-DN1}
\begin{cases}
\Delta \varphi=0 \text{ on }M,\\
J^*\varphi =\phi,\ i_N\varphi=0 \text{ on }\Sigma.
\end{cases}
\end{align}
Then the Dirichlet-to-Neumann operator $\mathcal{T}^{[p]}:\Omega^p(\Sigma)\to \Omega^p(\Sigma)$ in \cite{RS12} is
\begin{equation}
\mathcal{T}^{[p]}(\phi)=-i_Nd\varphi.
\end{equation}
By \cite[Theorem~11]{RS12}, the operator $\mathcal{T}^{[p]}$ possesses the following properties.
\begin{thm}
  Let $(M^{n+1},g)$ be a compact Riemannian manifold with boundary $\Sigma=\pt M$. The followings hold true.
  \begin{enumerate}
    \item The operator $\mathcal{T}^{[p]}$ is non-negative and self-adjoint, and the kernel of $\mathcal{T}^{[p]}$ is $\mathrm{ker}(\mathcal{T}^{[p]})=J^*\mathcal{H}^p_N(M)$, where $\mathcal{H}^p_N(M)$ is the space of harmonic fields with the Neumann boundary condition.
    \item The operator $\mathcal{T}^{[p]}$ is an elliptic pseudo-differential operator of the order one, and so it admits a non-decreasing sequence of eigenvalues with finite multiplicities
        \begin{equation*}
        0\leq \nu^{[p]}_1(\Sigma)\leq \nu^{[p]}_2(\Sigma)\leq \cdots \nearrow \infty,
        \end{equation*}
        with the eigenvalue $0$ repeated $b_p(M)=\mathrm{dim} \mathcal{H}^p_N(M)$ (the $p$th Betti number) times.
    \item The eigenvalues satisfy the variational principle
  \begin{align}
\nu_k^{[p]}(\Sigma)=\inf_{F_k }\sup_{\varphi \in F_k\setminus \{0\}}\frac{\int_M  |d  \varphi |^2+|\delta \varphi|^2 dv}{\int_\Sigma  |J^*\varphi |^2 da},
\end{align}
where $F_k$ ranges over $k$-dimensional subspaces in $\Omega^p(M)$ with the Neumann boundary condition (i.e., any $\varphi\in F_k$ satisfies $i_N \varphi=0$).
  \end{enumerate}
\end{thm}
We also present the spectrum of the operator $\mathcal{T}^{[p]}$ for the unit Euclidean ball.
\begin{exam}[{\cite{GM75},\cite{IT78},\cite[Prop.~7]{RS14},\cite[Thm.~8.1]{Kar19}}]\label{exam2}
  Let $M=\overline{B^{n+1}(1)}$ be the unit Euclidean ball. (i) Let $1\leq p\leq n-1$. The spaces $J^* H'_{l-1,p}$ and $J^*H''_{l,p}$ for $l\geq 1$ form the eigenbases of $\mathcal{T}^{[p]}$. Precisely, if $\phi\in J^* H'_{l-1,p}$, then \begin{equation*}
  \mathcal{T}^{[p]}\phi=\frac{(l+p-1)(n+2l+1)}{(n+2l-1)}\phi;
  \end{equation*}
  if $\phi\in J^* H''_{l,p}$, then $\mathcal{T}^{[p]}\phi=(p+l)\phi$. (ii) Let $p=n$. If $\phi\in J^* H'_{l-1,n}$, then \begin{equation*}
  \mathcal{T}^{[n]}\phi=\frac{(l+n-1)(n+2l+1)}{(n+2l-1)}\phi;
  \end{equation*}
  if $\phi$ is the volume form of $\SS^n$, then $\mathcal{T}^{[n]}\phi=(n+1)\phi$.
\end{exam}

Last we consider the boundary Hodge Laplacian $\Delta_\Sigma$ restricted to the space of co-closed differential forms
\begin{align*}
\Delta_\Sigma^{[p]}:c\mathcal{C}^p(\Sigma)\to c\mathcal{C}^p(\Sigma).
\end{align*}
This operator is a non-negative self-adjoint elliptic operator with eigenvalues
\begin{align*}
0\leq \lambda_1^{[p]}(\Sigma)\leq \lambda_2^{[p]}(\Sigma)\leq \dots \nearrow \infty.
\end{align*}
The eigenvalue $\lambda_k^{[p]}(\Sigma)$ admits the variational principle
\begin{align}
\lambda_k^{[p]}(\Sigma)=\inf_{G_k\subset c\mathcal{C}^p(\Sigma)}\sup_{\phi \in G_k\setminus \{0\}}\frac{\int_\Sigma  |d^\Sigma \phi |^2 da}{\int_\Sigma  |\phi |^2 da},
\end{align}
where $G_k$ ranges over $k$-dimensional subspaces of $c\mathcal{C}^p(\Sigma)$.

We describe the spectrum of the operator $\Delta_\Sigma^{[p]}$ on the whole space $\Omega^p(\Sigma)$ for the unit Euclidean ball as well.
\begin{exam}[{\cite{GM75},\cite{IT78},\cite[Prop.~7]{RS14},\cite[Thm.~8.1]{Kar19}}]\label{exam3}
  Let $M=\overline{B^{n+1}(1)}$ be the unit Euclidean ball. (i) Let $1\leq p\leq n-1$. The spaces $J^* H'_{l-1,p}$ and $J^*H''_{l,p}$ for $l\geq 1$ form the eigenbases of $\Delta_\Sigma^{[p]}$. Precisely, if $\phi\in J^* H'_{l-1,p}$, then
  \begin{equation*}
  \Delta_\Sigma^{[p]}\phi= (l+p-1)(n+l-p) \phi;
  \end{equation*}
  if $\phi\in J^* H''_{l,p}$, then
    \begin{equation*}
  \Delta_\Sigma^{[p]}\phi= (l+p)(n+l-p-1) \phi.
  \end{equation*}
  (ii) Let $p=n$. If $\phi\in J^* H'_{l-1,n}$, then
  \begin{equation*}
  \Delta_\Sigma^{[n]}\phi= (l+n-1)l \phi;
  \end{equation*}
  if $\phi$ is the volume form of $\SS^n$, then $\Delta_\Sigma^{[n]}\phi=0$.
\end{exam}

\subsection{The Pohozhaev-type identity}

For our main result we shall use the following Pohozhaev-type identity. This identity is essentially Theorem~5.2 in \cite{GKLP22}. We state it in our required form and include its short proof for readers' convenience.
\begin{thm}[{\cite[Thm.~5.2]{GKLP22}}]\label{prop-Pohozhaev}
Let $(M^{n+1},g)$ be a smooth compact connected oriented Riemannian manifold with boundary $\Sigma=\pt M$. Let $F$ be a Lipschitz vector field on $M$, and $\varphi\in \Omega^p(M)$ a differential form. Then
  \begin{align}\label{eq-Poh}
  &\int_M |d\varphi|^2 \mathrm{div} F dv=-\int_\Sigma |d\varphi|^2 \langle F,N\rangle da-2\int_M \langle i_F d\varphi,\delta d\varphi\rangle dv\nonumber \\
  &+2\int_\Sigma \langle J^* i_F d\varphi, i_N d\varphi\rangle da+2\int_M \langle \nabla F(d\varphi),d\varphi\rangle dv.
  \end{align}
\end{thm}
\begin{proof}
  We start with the equality
  \begin{equation*}
  |d\varphi|^2\mathrm{div}F=\mathrm{div}(|d\varphi|^2 F)-\nabla_F|d\varphi|^2=\mathrm{div}(|d\varphi|^2 F)-2\langle \nabla_F d\varphi,d\varphi\rangle.
  \end{equation*}
  Next we use the computation result in Proposition~\ref{prop-computation} below
  \begin{align*}
  d(i_F\omega)=-i_F(d\omega)+\nabla_F\omega+\nabla F(\omega).
  \end{align*}
  So we may replace $\nabla_F d\varphi=di_Fd\varphi-\nabla F(d\varphi)$ to obtain
  \begin{equation*}
  |d\varphi|^2\mathrm{div}F= \mathrm{div}(|d\varphi|^2 F)-2\langle di_Fd\varphi,d\varphi\rangle+2\langle \nabla F(d\varphi),d\varphi\rangle.
  \end{equation*}
  Then integrating it over $M$ and using Stokes formula, we find
  \begin{align*}
  &\int_M |d\varphi|^2 \mathrm{div} Fdv=-\int_\Sigma |d\varphi|^2 \langle F,N\rangle da-2\int_M \langle i_F d\varphi,\delta d\varphi\rangle dv\\
  &+2\int_\Sigma \langle J^* i_F d\varphi, i_N d\varphi\rangle da +2\int_M \langle \nabla F(d\varphi),d\varphi\rangle dv,
  \end{align*}
  which is the conclusion.
\end{proof}

\subsection{The weight function and its approximation}

In the proofs of our eigenvalue estimates we also need a suitably chosen weight function $f_\varepsilon$.

Define the distance function to the boundary $\Sigma$ by
\begin{equation*}
\rho=\rho(x)={\rm dist}(x,  \Sigma).
\end{equation*}
The distance function $\rho$ is smooth away from the cut locus ${\rm Cut}(\Sigma)$ of $\Sigma$. Recall that ${\rm Cut}(\Sigma)$ is defined to be the set of all cut points and a cut point is the first point on a normal geodesic initiating from the boundary $\Sigma$ at which this geodesic fails to minimize  uniquely for the distance function $\rho$. In other words, for $x\in \Sigma$, consider the arc-length parametrized geodesic $\gamma_x(t)=\exp_x(tN(x))$ ($t\geq 0$). Then
$\gamma_x(t_0)\in {\rm Cut}(\Sigma)$ for
\begin{equation*}
t_0=t_0(x)=\sup\{t>0:\mathrm{dist}(\gamma_x(t),\Sigma)=t\}.
\end{equation*}

The set ${\rm Cut}(\Sigma)$ is known to have zero $(n+1)$-dimensional Hausdorff measure; see e.g. \cite[Thm.~B]{IT01}. In addition, if the Ricci curvature of $M$ is non-negative and the mean curvature $H$ of $\Sigma$ satisfies $H\geq c>0$, we have
\begin{eqnarray}\label{dist-max}
\rho_{\max}=\max_{ M} \rho\leq \frac{1}{c}.
\end{eqnarray}
See e.g. \cite{Li14}.

Next we define a weight function as
\begin{align}\label{f-def}
f(x)=f(\rho(x))=\rho(x)-\frac{c}{2}\rho(x)^2.
\end{align}
By the Hessian comparison theorem and a smoothing approximation, we have the following result.
\begin{prop}[\cite{XX23}]\label{prop-approximation}
Assume that the sectional curvature of $M$ is non-negative and the second fundamental form $h$ of the boundary satisfies $h\geq cg_\Sigma$ for some constant $c>0$. Fix any neighborhood $\cal{C}$ of ${\rm Cut}(\Sigma)$ in $M$. Then for any $\varepsilon>0$, there exists a smooth nonnegative function ${f}_\varepsilon$ on $M$ such that $f_\varepsilon=f$ on $M\setminus \cal{C}$ and
\begin{equation}
\nabla^2 (-{f}_\varepsilon)\geq (c-\varepsilon)g.
\end{equation}
Moreover, $f_\varepsilon$ converges uniformly to $f$ on $M$ as $\varepsilon \to 0$.
\end{prop}
The first part of Proposition~\ref{prop-approximation} is just Proposition~3.3 in \cite{XX23}. The uniform convergence is stated in the proof of Proposition~4.2 in \cite{XX23}.

\section{The weighted Reilly formula for differential forms}\label{sec3}

In this section, we give the proof of Theorem~\ref{thm-Reilly}. For simplicity we shall omit the volume and area elements in the integrals, and we also adopt the Einstein convention for summation on indices. First by the Bochner formula, we see
\begin{align*}
\frac{1}{2}\Delta (|\omega|^2)&=-\frac{1}{2}\langle \omega,\omega\rangle_{e_i e_i}=-\langle \nabla_{e_i}\omega,\omega\rangle_{e_i}\\
&=\langle \nabla^*_{e_i}\nabla_{e_i} \omega,\omega\rangle -\langle \nabla_{e_i}\omega,\nabla_{e_i}\omega\rangle\\
&=\langle \Delta \omega ,\omega\rangle-\langle W_{M}^{[p]}(\omega),\omega\rangle -|\nabla \omega|^2.
\end{align*}

Now we consider a smooth weight function $f\in C^\infty(M)$. Then we obtain
\begin{align*}
\frac{1}{2}\Delta(f|\omega|^2)&=\frac{1}{2}\Delta f\cdot |\omega|^2-\langle \omega,\omega\rangle_{\nabla f}+f\cdot \frac{1}{2}\Delta (|\omega|^2)\\
&=\frac{1}{2}\Delta f\cdot |\omega|^2-\langle \omega,\omega\rangle_{\nabla f}+f\left(\langle \Delta \omega ,\omega\rangle-\langle W_{M}^{[p]}(\omega),\omega\rangle -|\nabla \omega|^2\right).
\end{align*}
Integrating it over $M$, we get
\begin{align*}
\int_M \frac{1}{2}\Delta(f|\omega|^2)=&\int_M \left(\frac{1}{2}\Delta f\cdot |\omega|^2-\langle \omega,\omega\rangle_{\nabla f}-f\left( \langle W_{M}^{[p]}(\omega),\omega\rangle +|\nabla \omega|^2\right)\right)\\
&+\int_M f \langle \Delta \omega ,\omega\rangle.
\end{align*}
Next we consider $\int_M \langle \delta d \omega, f\omega \rangle$ and $\int_M \langle d\delta \omega, f\omega \rangle$ in $\int_M f \langle \Delta \omega ,\omega\rangle$ separately.

\subsection{The calculation of $\int_M \langle \delta d \omega, f\omega \rangle$.} First we deduce
\begin{align*}
\int_M \langle \delta d \omega, f\omega \rangle=&\int_M \langle   d \omega, df\wedge \omega+fd\omega \rangle+\int_\Sigma \langle i_N(d\omega),J^*(f\omega) \rangle \\
=&\int_M \langle    \omega, \delta(df\wedge \omega) \rangle-\int_\Sigma \langle J^*\omega , i_N (df\wedge \omega)\rangle\\
& +\int_M f|d\omega|^2+\int_\Sigma \langle i_N(d\omega),J^*(f\omega) \rangle.
\end{align*}
Next we expand the term $\delta (df\wedge \omega)$.
\begin{prop}\label{prop-delta}
We have
\begin{equation}
\delta (df\wedge \omega)=\Delta f \cdot \omega-\nabla_{\nabla f}\omega +\nabla^2 f(\omega)-df\wedge \delta \omega.
\end{equation}
\end{prop}
\begin{proof}
Let $\{e_i\}_{i=1}^{n+1}$ be a local orthonormal frame for $\mathfrak{X}(M)$, and $X_j\in \mathfrak{X}(M)$, $j=1,2,\dots, p$, be smooth vector fields. Moreover, we may assume that $e_i$'s and $X_j$'s are geodesic at the point of computation. Then we derive
\begin{align*}
&\delta (df\wedge \omega)(X_1,\dots,X_p)=-i_{e_i}(\nabla_{e_i}(df\wedge \omega))(X_1,\dots,X_p)\\
&=-(\nabla_{e_i}(df\wedge \omega))(e_i,X_1,\dots,X_p)\\
&=-e_i((df\wedge \omega)(e_i,X_1,\dots,X_p))\\
&=-e_i\left(\langle \nabla f,e_i\rangle \omega(X_1,\dots,X_p)+\sum_{k=1}^p (-1)^{k}\langle \nabla f,X_k\rangle \omega (e_i,X_1,\dots,\hat{X_k},\dots, X_p) \right)\\
&=(\Delta f \cdot \omega-\nabla_{\nabla f}\omega+(\nabla^2 f)\omega)(X_1,\dots,X_p)\\
&\quad +\sum_{k=1}^p (-1)^{k+1}\langle \nabla f,X_k\rangle \nabla_{e_i}\omega (e_i,X_1,\dots,\hat{X_k},\dots, X_p),
\end{align*}
where $\hat{X_k}$ means that the argument $X_k$ disappears.

Noting
\begin{align*}
&\sum_{k=1}^p (-1)^{k+1}\langle \nabla f,X_k\rangle \nabla_{e_i}\omega (e_i,X_1,\dots,\hat{X_k},\dots, X_p)\\
&=\sum_{k=1}^p (-1)^{k}\langle \nabla f,X_k\rangle \delta\omega ( X_1,\dots,\hat{X_k},\dots, X_p)\\
&=-df\wedge \delta\omega(X_1,\dots,X_p),
\end{align*}
we may finish the proof.
\end{proof}
Hence we get
\begin{align*}
&\int_M \langle \delta d \omega, f\omega \rangle=\int_M \langle    \omega, \Delta f \cdot \omega-\nabla_{\nabla f}\omega+(\nabla^2 f)\omega-df\wedge \delta\omega\rangle\\
&-\int_\Sigma \langle J^*\omega , i_N (df\wedge \omega)\rangle +\int_M f|d\omega|^2+\int_\Sigma \langle i_N(d\omega),J^*(f\omega) \rangle\\
&=\int_M \Delta f | \omega|^2-\langle \omega,\nabla_{\nabla f}\omega\rangle +\langle \omega,(\nabla^2 f)\omega\rangle-\langle \omega, df\wedge \delta \omega\rangle\\
&-\int_\Sigma \langle J^*\omega , i_N (df\wedge \omega)\rangle +\int_M f|d\omega|^2+\int_\Sigma \langle i_N(d\omega),J^*(f\omega) \rangle.
\end{align*}
So we obtain
\begin{align*}
&\int_M \frac{1}{2}\Delta(f|\omega|^2)=\int_M \left(\frac{1}{2}\Delta f\cdot |\omega|^2-\langle \omega,\omega\rangle_{\nabla f}-f\left( \langle W_{M}^{[p]}(\omega),\omega\rangle +|\nabla \omega|^2\right)\right)\\
&+\int_M \langle d\delta \omega, f\omega \rangle+\int_M \Delta f | \omega|^2-\langle \omega,\nabla_{\nabla f}\omega\rangle +\langle \omega,(\nabla^2 f)\omega\rangle-\langle \omega, df\wedge \delta \omega\rangle\\
&-\int_\Sigma \langle J^*\omega , i_N (df\wedge \omega)\rangle +\int_M f|d\omega|^2+\int_\Sigma \langle i_N(d\omega),J^*(f\omega) \rangle.
\end{align*}
Next we compute the term $-\int_M \langle \omega, df\wedge \delta \omega\rangle$ as follows.
\begin{align*}
&-\int_M \langle \omega, df\wedge \delta \omega\rangle=-\int_M \langle \omega, d(f \delta \omega)-fd\delta \omega \rangle\\
&=-\int_M f|\delta \omega|^2+\int_\Sigma \langle i_N \omega, J^*(f\delta \omega)\rangle +\int_M f\langle d\delta \omega, \omega\rangle.
\end{align*}
As a result, we get
\begin{align}\label{eq-first-formula}
&\int_M \frac{1}{2}\Delta(f|\omega|^2)=\int_M \left(\frac{1}{2}\Delta f\cdot |\omega|^2-\langle \omega,\omega\rangle_{\nabla f}-f\left( \langle W_{M}^{[p]}(\omega),\omega\rangle +|\nabla \omega|^2\right)\right)\nonumber\\
&+2\int_M \langle d\delta \omega, f\omega \rangle+\int_M \Delta f | \omega|^2-\langle \omega,\nabla_{\nabla f}\omega\rangle +\langle \omega,(\nabla^2 f)\omega\rangle\nonumber\\
&-\int_M f|\delta \omega|^2+\int_\Sigma \langle i_N \omega, J^*(f\delta \omega)\rangle\nonumber\\
&-\int_\Sigma \langle J^*\omega , i_N (df\wedge \omega)\rangle +\int_M f|d\omega|^2+\int_\Sigma \langle i_N(d\omega),J^*(f\omega) \rangle.
\end{align}

\subsection{The calculation of  $\int_M \langle d\delta \omega, f\omega \rangle$.} For this term, we derive
\begin{align*}
\int_M \langle d\delta \omega, f\omega \rangle&=\int_M \langle \delta \omega, \delta(f\omega) \rangle-\int_\Sigma \langle J^*(\delta \omega),i_N(f\omega)\rangle.
\end{align*}
Note that
\begin{align*}
\delta (f\omega)=-i_{e_i}(\nabla_{e_i}(f\omega))=-i_{\nabla f}\omega+f\delta \omega.
\end{align*}
So we get
\begin{align*}
\int_M \langle d\delta \omega, f\omega \rangle=&\int_M \langle \delta \omega, -i_{\nabla f}\omega \rangle+\int_M f|\delta \omega|^2-\int_\Sigma \langle J^*(\delta \omega),i_N(f\omega)\rangle\\
=&-\int_M \langle  \omega, d(i_{\nabla f}\omega) \rangle-\int_\Sigma \langle i_N\omega, J^*(i_{\nabla f}\omega)\rangle\\
&+\int_M f|\delta \omega|^2-\int_\Sigma \langle J^*(\delta \omega),i_N(f\omega)\rangle.
\end{align*}
Next we compute the differential $d(i_{\nabla f}\omega)$, which results from the following.
\begin{prop}\label{prop-computation}
For a Lipschitz vector field $F$ on $M$, we have
\begin{equation}
d(i_F\omega)=-i_F(d\omega)+\nabla_F\omega+\nabla F(\omega).
\end{equation}
\end{prop}
\begin{proof}
As before, let $\{e_i\}_{i=1}^{n+1}$ be a local orthonormal frame for $\mathfrak{X}(M)$, and $X_j\in \mathfrak{X}(M)$, $j=1,2,\dots, p$, be smooth vector fields. Moreover, we may assume that $e_i$'s and $X_j$'s are geodesic at the point of computation. Denote by $\{e_i^*\}_{i=1}^{n+1}$ the dual frame of $\{e_i\}_{i=1}^{n+1}$. We may deduce
\begin{align*}
&d(i_F\omega)(X_1,\dots,X_p)=e_i^*\wedge \nabla_{e_i}(i_F\omega)(X_1,\dots,X_p)\\
&=\sum_{k=1}^p(-1)^{k+1}\langle e_i,X_k\rangle \nabla_{e_i}(i_F\omega)(X_1,\dots,\hat{X_k},\dots,X_p)\\
&=\sum_{k=1}^p(-1)^{k+1}\langle e_i,X_k\rangle  {e_i}( \omega (F,X_1,\dots,\hat{X_k},\dots,X_p))\\
&=\sum_{k=1}^p(-1)^{k+1}\langle e_i,X_k\rangle  \bigg(\nabla_{e_i} \omega (F,X_1,\dots,\hat{X_k},\dots,X_p) \\
&\quad +  \omega (\nabla_{e_i}F,X_1,\dots,\hat{X_k},\dots,X_p)\bigg)\\
&=-\bigg(\sum_{k=1}^p(-1)^{k}\langle e_i,X_k\rangle  \nabla_{e_i} \omega (F,X_1,\dots,\hat{X_k},\dots,X_p) \\
&\quad +\langle e_i,F\rangle  \nabla_{e_i} \omega ( X_1,\dots,X_p)-\langle e_i,F\rangle  \nabla_{e_i} \omega ( X_1,\dots,X_p)\bigg)\\
&\quad +\sum_{k=1}^{p}\omega(X_1,\dots, \nabla F(X_k),\dots, X_p)\\
&=(-i_F(d\omega)+\nabla_F\omega+\nabla F(\omega))(X_1,\dots,X_p),
\end{align*}
which is the conclusion.
\end{proof}
Therefore we derive
\begin{align}\label{eq-second-formula}
&\int_M \langle d\delta \omega, f\omega \rangle=\int_M \left(\langle  \omega,  i_{\nabla f}(d\omega) \rangle-\langle \omega, \nabla_{\nabla f}\omega\rangle-\langle \omega, \nabla^2 f(\omega)\rangle\right)\nonumber\\
&-\int_\Sigma \langle i_N\omega, J^*(i_{\nabla f}\omega)\rangle+\int_M f|\delta \omega|^2-\int_\Sigma \langle J^*(\delta \omega),i_N(f\omega)\rangle.
\end{align}

\subsection{Further computation.} Next, inserting \eqref{eq-second-formula} into \eqref{eq-first-formula} we obtain
\begin{align*}
\int_M \frac{1}{2}\Delta(f|\omega|^2)&=\int_M \left(\frac{1}{2}\Delta f\cdot |\omega|^2-\langle \omega,\omega\rangle_{\nabla f}-f\left( \langle W_{M}^{[p]}(\omega),\omega\rangle +|\nabla \omega|^2\right)\right)\\
&+2\int_M \left(\langle  \omega,  i_{\nabla f}(d\omega) \rangle-\langle \omega, \nabla_{\nabla f}\omega\rangle-\langle \omega, \nabla^2 f(\omega)\rangle\right)\\
&-2\int_\Sigma \langle i_N\omega, J^*(i_{\nabla f}\omega)\rangle+2\int_M f|\delta \omega|^2-2\int_\Sigma \langle J^*(\delta \omega),i_N(f\omega)\rangle\\
&+\int_M \Delta f | \omega|^2-\langle \omega,\nabla_{\nabla f}\omega\rangle +\langle \omega,(\nabla^2 f)\omega\rangle \\
&-\int_\Sigma \langle J^*\omega , i_N (df\wedge \omega)\rangle +\int_M f|d\omega|^2+\int_\Sigma \langle i_N(d\omega),J^*(f\omega) \rangle\\
&-\int_M f|\delta \omega|^2+\int_\Sigma \langle i_N \omega, J^*(f\delta \omega)\rangle.
\end{align*}
Rearranging the terms, we see
\begin{align*}
&\int_M \frac{1}{2}\Delta(f|\omega|^2)=\int_M \left(\frac{3}{2}\Delta f\cdot |\omega|^2-\frac{5}{2}\langle \omega,\omega\rangle_{\nabla f}-f\left( \langle W_{M}^{[p]}(\omega),\omega\rangle +|\nabla \omega|^2\right)\right)\\
&+\int_M \left(2\langle  \omega,  i_{\nabla f}(d\omega) \rangle -\langle \omega, \nabla^2 f(\omega)\rangle\right)+\int_M f(|\delta \omega|^2+|d\omega|^2)\\
&-2\int_\Sigma \langle i_N\omega, J^*(i_{\nabla f}\omega)\rangle-\int_\Sigma \langle J^*(\delta \omega),i_N(f\omega)\rangle\\
&-\int_\Sigma \langle J^*\omega , i_N (df\wedge \omega)\rangle + \int_\Sigma \langle i_N(d\omega),J^*(f\omega) \rangle.
\end{align*}
Next we compute the term
\begin{align*}
&\int_M \langle \omega,\omega\rangle_{\nabla f}=\int_M \mathrm{div}(|\omega|^2 \nabla f)+|\omega|^2 \Delta f=-\int_\Sigma |\omega|^2 f_N+\int_M  |\omega|^2 \Delta f.
\end{align*}
On the other hand, we get
\begin{align*}
&\int_M \frac{1}{2}\Delta(f|\omega|^2)=\int_\Sigma \frac{1}{2} f_N|\omega|^2+f\langle \nabla_N\omega,\omega\rangle\\
&=\int_\Sigma \frac{1}{2} f_N|\omega|^2+f(\langle J^*(\nabla_N \omega), J^*\omega\rangle +\langle i_N(\nabla_N \omega),i_N\omega\rangle).
\end{align*}
Consequently, we obtain
\begin{align*}
&\int_\Sigma  f(\langle J^*(\nabla_N \omega), J^*\omega\rangle +\langle i_N(\nabla_N \omega),i_N\omega\rangle)\\
&=\int_M \left(-\Delta f\cdot |\omega|^2-f\left( \langle W_{M}^{[p]}(\omega),\omega\rangle +|\nabla \omega|^2\right)\right)\\
&+\int_M \left(2\langle  \omega,  i_{\nabla f}(d\omega) \rangle -\langle \omega, \nabla^2 f(\omega)\rangle\right)+\int_M f(|\delta \omega|^2+|d\omega|^2)\\
&+2\int_\Sigma |\omega|^2 f_N-2\int_\Sigma \langle i_N\omega, J^*(i_{\nabla f}\omega)\rangle-\int_\Sigma \langle J^*(\delta \omega),i_N(f\omega)\rangle\\
&-\int_\Sigma \langle J^*\omega , i_N (df\wedge \omega)\rangle + \int_\Sigma \langle i_N(d\omega),J^*(f\omega) \rangle.
\end{align*}

Now we recall the formulas from \cite{RS11}.
\begin{prop}[{\cite[Lem.~18 (ii)]{RS11}}]
There hold
\begin{align}
\delta^\Sigma (J^*\omega)&=J^*(\delta \omega)+i_N(\nabla_N\omega)+S^{[p-1]}(i_N\omega)-nHi_N\omega,\label{eq3.5}\\
d^\Sigma i_N\omega&=-i_Nd\omega+J^*(\nabla_N\omega)-S^{[p]}(J^*\omega).\label{eq3.6}
\end{align}
\end{prop}
So we substitute $J^*(\delta \omega)$ from \eqref{eq3.5} and $J^*(\nabla_N\omega)$ from \eqref{eq3.6} into the equality to get
\begin{align*}
&\int_\Sigma  f \langle d^\Sigma i_N\omega+i_Nd\omega+S^{[p]}(J^*\omega), J^*\omega\rangle +\int_\Sigma f\langle i_N(\nabla_N \omega),i_N\omega\rangle \\
&=\int_M \left(-\Delta f\cdot |\omega|^2-f\left( \langle W_{M}^{[p]}(\omega),\omega\rangle +|\nabla \omega|^2\right)\right)\\
&+\int_M \left(2\langle  \omega,  i_{\nabla f}(d\omega) \rangle -\langle \omega, \nabla^2 f(\omega)\rangle\right)+\int_M f(|\delta \omega|^2+|d\omega|^2)\\
&+2\int_\Sigma |\omega|^2 f_N-2\int_\Sigma \langle i_N\omega, J^*(i_{\nabla f}\omega)\rangle\\
&-\int_\Sigma \langle J^*\omega , i_N (df\wedge \omega)\rangle + \int_\Sigma \langle i_N(d\omega),J^*(f\omega) \rangle\\
&+\int_\Sigma \langle -\delta^\Sigma (J^*\omega)+i_N(\nabla_N\omega)+S^{[p-1]}(i_N\omega)-nHi_N\omega,i_N(f\omega)\rangle.
\end{align*}
Simplifying it, we derive
\begin{align*}
&0=\int_M \left(-\Delta f\cdot |\omega|^2-f\left( \langle W_{M}^{[p]}(\omega),\omega\rangle +|\nabla \omega|^2\right)\right)\\
&+\int_M \left(2\langle  \omega,  i_{\nabla f}(d\omega) \rangle -\langle \omega, \nabla^2 f(\omega)\rangle\right)+\int_M f(|\delta \omega|^2+|d\omega|^2)\\
&+2\int_\Sigma |\omega|^2 f_N-2\int_\Sigma \langle i_N\omega, J^*(i_{\nabla f}\omega)\rangle-\int_\Sigma  f \langle d^\Sigma i_N\omega+S^{[p]}(J^*\omega), J^*\omega\rangle  \\
&-\int_\Sigma \langle J^*\omega , i_N (df\wedge \omega)\rangle +\int_\Sigma \langle -\delta^\Sigma (J^*\omega)+S^{[p-1]}(i_N\omega)-nHi_N\omega,i_N(f\omega)\rangle.
\end{align*}
Noting
\begin{align*}
-\int_\Sigma   \langle d^\Sigma i_N\omega , fJ^*\omega\rangle &=-\int_\Sigma   \langle  i_N\omega , \delta^\Sigma (fJ^*\omega)\rangle\\
&=-\int_\Sigma   \langle  i_N\omega , -i_{\nabla^\Sigma f}J^*\omega+f\delta^\Sigma J^*\omega \rangle,
\end{align*}
we finally get
\begin{align*}
0=&\int_M \left(-\Delta f\cdot |\omega|^2-f\left( \langle W_{M}^{[p]}(\omega),\omega\rangle +|\nabla \omega|^2\right)\right)\\
&+\int_M \left(2\langle  \omega,  i_{\nabla f}(d\omega) \rangle -\langle \omega, \nabla^2 f(\omega)\rangle\right)+\int_M f(|\delta \omega|^2+|d\omega|^2)\\
&+2\int_\Sigma |\omega|^2 f_N-2\int_\Sigma \langle i_N\omega, J^*(i_{\nabla f}\omega)\rangle-\int_\Sigma  f \langle  S^{[p]}(J^*\omega), J^*\omega\rangle \\
& +\int_\Sigma   \langle  i_N\omega ,  i_{\nabla^\Sigma f}J^*\omega \rangle
-\int_\Sigma \langle J^*\omega , i_N (df\wedge \omega)\rangle \\
&+\int_\Sigma \langle -2\delta^\Sigma (J^*\omega)+S^{[p-1]}(i_N\omega)-nHi_N\omega,i_N(f\omega)\rangle.
\end{align*}
Recall
\begin{equation}
\mathcal{B}(\omega,\omega)=\langle  S^{[p]}(J^*\omega), J^*\omega\rangle+nH|  i_N\omega |^2-\langle S^{[p-1]}(i_N\omega) ,i_N \omega\rangle.
\end{equation}
We may get
\begin{align*}
&\int_M \left( \Delta f\cdot |\omega|^2+f\left( \langle W_{M}^{[p]}(\omega),\omega\rangle +|\nabla \omega|^2\right)\right)\\
&-\int_M \left(2\langle  \omega,  i_{\nabla f}(d\omega) \rangle -\langle \omega, \nabla^2 f(\omega)\rangle\right)-\int_M f(|\delta \omega|^2+|d\omega|^2)\\
&=2\int_\Sigma |\omega|^2 f_N-2\int_\Sigma \langle i_N\omega, J^*(i_{\nabla f}\omega)\rangle+\int_\Sigma   \langle  i_N\omega ,  i_{\nabla^\Sigma f}J^*\omega \rangle\\
&-\int_\Sigma \langle J^*\omega , i_N (df\wedge \omega)\rangle -2\int_\Sigma \langle \delta^\Sigma (J^*\omega),i_N(f\omega)\rangle-\int_\Sigma f\mathcal{B}(\omega,\omega).
\end{align*}

Next we note
\begin{align*}
i_{\nabla f}\omega&=f_N i_N\omega +i_{\nabla^\Sigma f}\omega,\\
i_N(df\wedge \omega)&= f_NJ^*\omega-d^\Sigma f\wedge i_N\omega,\\
|\omega|^2&=|J^*\omega|^2+|i_N\omega|^2.
\end{align*}
Therefore we get
\begin{align*}
&\int_M \left( \Delta f\cdot  |\omega |^2+f\left( \langle W_{M}^{[p]}(\omega),\omega\rangle + |\nabla \omega |^2\right)\right)\\
&-\int_M \left(2\langle  \omega,  i_{\nabla f}(d\omega) \rangle -\langle \omega, \nabla^2 f(\omega)\rangle\right)-\int_M f( |\delta \omega |^2+ |d\omega |^2)\\
&=\int_\Sigma f_N |J^*\omega |^2 -\int_\Sigma   \langle  i_N\omega ,  i_{\nabla^\Sigma f}J^*\omega \rangle+\int_\Sigma \langle J^*\omega , d^\Sigma f\wedge i_N\omega\rangle\\
& -2\int_\Sigma \langle \delta^\Sigma (J^*\omega),i_N(f\omega)\rangle-\int_\Sigma f\mathcal{B}(\omega,\omega).
\end{align*}
Next we invoke a basic fact for differential forms, for $\varphi\in \Omega^q(M)$, $\psi\in \Omega^{q-1}(M)$, and $X\in \mathfrak{X}(M)$ with the dual $1$-form $X^*$,
\begin{equation*}
\langle \varphi,X^*\wedge \psi\rangle=\langle i_X\varphi,\psi\rangle,
\end{equation*}
which can be verified straightforwardly. Thus
\begin{align*}
&\langle J^*\omega , d^\Sigma f\wedge i_N\omega\rangle=\langle  i_N\omega ,  i_{\nabla^\Sigma f}J^*\omega \rangle,
\end{align*}
and we get
\begin{align*}
&\int_M \left( \Delta f\cdot  |\omega |^2+f\left( \langle W_{M}^{[p]}(\omega),\omega\rangle + |\nabla \omega |^2\right)\right)\\
&-\int_M \left(2\langle  \omega,  i_{\nabla f}(d\omega) \rangle -\langle \omega, \nabla^2 f(\omega)\rangle\right)-\int_M f( |\delta \omega |^2+| d\omega |^2)\\
&=\int_\Sigma f_N  |J^*\omega |^2  -2\int_\Sigma \langle \delta^\Sigma (J^*\omega),i_N(f\omega)\rangle-\int_\Sigma f\mathcal{B}(\omega,\omega),
\end{align*}
which is the desired conclusion in Theorem~\ref{thm-Reilly}.

\section{A sharp lower bound for the first non-zero Steklov eigenvalue}\label{sec4}

In this section we prove Theorem~\ref{thm-bound}. Let $\varphi$ be an eigenform corresponding to the eigenvalue $\sigma$. Then choosing $\omega=d\varphi$ and $f=f_\varepsilon$ in our weighted Reilly formula, and noting
\begin{align*}
&\delta d\varphi=\delta d\varphi + d\delta \varphi=\Delta \varphi=0,\text{ on }M,\\
&\qquad f_\varepsilon=0,\  \nabla f_\varepsilon=N, \text{ on }\Sigma,
\end{align*}
we get
\begin{align}\label{eq-comparison}
\int_\Sigma  |J^*d\varphi |^2 da=& \int_M \langle d\varphi, \nabla^2 f_\varepsilon (d\varphi)\rangle dv+\int_M \Delta f_\varepsilon |d\varphi |^2 dv\nonumber \\
&+\int_M f_\varepsilon (|\nabla d\varphi|^2+\langle W_M^{[p+1]}(d\varphi),d\varphi\rangle)dv.
\end{align}

On the other hand, by the Pohozhaev-type identity \eqref{eq-Poh} with $F=\nabla f_\varepsilon$ (noting $F=N$ on $\Sigma$ and $\delta d\varphi=0$ on $M$), we see
\begin{align}
&\int_M  |d\varphi |^2 \Delta f_\varepsilon dv+2\int_M \langle d\varphi, \nabla^2 f_\varepsilon (d\varphi)\rangle dv\nonumber \\
&\qquad =\int_\Sigma  |J^* d\varphi |^2 da-\int_\Sigma  |i_N d\varphi |^2 da.
\end{align}
Adding these two identities and using Proposition~\ref{prop-approximation}, we derive
\begin{align*}
\int_\Sigma |i_N d\varphi |^2 da=& -\int_M \langle d\varphi, \nabla^2 f_\varepsilon (d\varphi)\rangle dv\\
&+\int_M f_\varepsilon (|\nabla d\varphi|^2+\langle W_M^{[p+1]}(d\varphi),d\varphi\rangle) dv\\
\geq &(p+1)(c-\varepsilon)\int_M  | d\varphi |^2 dv\\
&+\int_M f_\varepsilon (|\nabla d\varphi|^2+\langle W_M^{[p+1]}(d\varphi),d\varphi\rangle)dv.
\end{align*}
Letting $\varepsilon \to 0^+$ and using Proposition~\ref{prop-approximation} again, we get
\begin{align*}
\int_\Sigma  |i_N d\varphi |^2da&\geq (p+1)c\int_M  | d\varphi |^2 dv+\int_M f  (|\nabla d\varphi|^2+\langle W_M^{[p+1]}(d\varphi),d\varphi\rangle) dv.
\end{align*}
Finally note that
\begin{align}
\int_\Sigma  |i_N d\varphi |^2 da&=\sigma^2 \int_\Sigma  |\phi |^2 da,\\
\int_M  | d\varphi |^2 dv&=\sigma \int_\Sigma  |\phi |^2 da.
\end{align}
So we get the desired inequality.

For a Euclidean ball with the radius $1/c$, it follows from Example~\ref{exam1} that $\sigma_1^{[p]}=(p+1)c$ for $1\leq p\leq n$ and we get the equality. Conversely, when the equality holds, i.e., $\sigma=(p+1)c$, we readily see that
\begin{align*}
\int_M f  (|\nabla d\varphi|^2+\langle W_M^{[p+1]}(d\varphi),d\varphi\rangle)dv=0,
\end{align*}
which implies
\begin{align*}
\nabla d\varphi=0 \text{ and } \langle W_M^{[p+1]}(d\varphi),d\varphi\rangle=0\text{ on }M.
\end{align*}
Moreover, since $\varphi$ is the eigenform corresponding to the eigenvalue $\sigma=(p+1)c$, it does not belong to $\mathcal{H}^p_D(M)$. So $\varphi$ is a non-trivial solution to \eqref{eq-rigidity} and we finish the proof.

\section{Comparison between Steklov eigenvalues and boundary Hodge Laplacian eigenvalues}\label{sec5}

In this section we prove Theorem~\ref{thm-comparison}. First for $\forall \varphi\in \Omega^p(M)$ satisfying the equation~\eqref{eq-DN}, by the equality \eqref{eq-comparison} and Proposition~\ref{prop-approximation}, we get
\begin{align}\label{eq-5.1}
\int_\Sigma  |J^*d\varphi |^2 da \geq &(n-p)(c-\varepsilon)\int_M   |d\varphi |^2 dv\nonumber \\
&+\int_M f_\varepsilon (|\nabla d\varphi|^2+\langle W_M^{[p+1]}(d\varphi),d\varphi\rangle) dv.
\end{align}
Here the inequality is due to the following pointwise estimate. At a point $x\in M$, let $\{\gamma_i\}_{i=1}^{n+1}$ be the eigenvalues of $-\nabla^2 f_\varepsilon$ with corresponding unit eigenvectors $\{e_i\}_{i=1}^{n+1}$ in $T_xM$. Then we find
\begin{align*}
&\quad \Delta f_\varepsilon |d\varphi|^2+\langle d\varphi, \nabla^2 f_\varepsilon (d\varphi)\rangle\\
&=\sum_{i_1<i_2<\cdots<i_{p+1}}\left(\sum_{i=1}^{n+1}\gamma_i-(\gamma_{i_1}+\gamma_{i_2}+\cdots+\gamma_{i_{p+1}})\right) ((d\varphi)_{i_1i_2\dots i_{p+1}})^2\\
&=\sum_{i_1<i_2<\cdots<i_{p+1}}\left(\sum_{i\neq i_1,i_2,\dots,i_{p+1}}\gamma_i\right) ((d\varphi)_{i_1i_2\dots i_{p+1}})^2\\
&\geq (n-p)(c-\varepsilon)\sum_{i_1<i_2<\cdots<i_{p+1}} ((d\varphi)_{i_1i_2\dots i_{p+1}})^2\\
&= (n-p)(c-\varepsilon)|d\varphi|^2.
\end{align*}

Next letting $\varepsilon \to 0^+$ in the inequality \eqref{eq-5.1} and using Proposition~\ref{prop-approximation}, we get
\begin{align}\label{eq-5.2}
\int_\Sigma  |J^*d\varphi |^2da &\geq  (n-p)c\int_M    |d\varphi |^2 dv+\int_M f (|\nabla d\varphi|^2+\langle W_M^{[p+1]}(d\varphi),d\varphi\rangle) dv\nonumber\\
&\geq (n-p)c\int_M    |d\varphi |^2 dv.
\end{align}

Now let $\{\phi_i\}_{i=1}^k\subset c\mathcal{C}^p(\Sigma)$ be orthonormal eigenforms of the Hodge Laplacian $\Delta_\Sigma$ corresponding to eigenvalues $\{\lambda_i^{[p]}\}_{i=1}^k$. Let $\{\varphi_i\}_{i=1}^k\subset \Omega^p(M)$ be a set of linearly independent solutions to the equation~\eqref{eq-DN} satisfying $J^*\varphi_i=\phi_i$ on $\Sigma$. Then for any $a_i\in\R$, $i=1,2,\dots,k$ (not all being zero), using the inequality~\eqref{eq-5.2} we get
\begin{align*}
\sigma^{[p]}_k&\leq \frac{\int_M  |d  \sum_{i=1}^k a_i\varphi_i |^2dv}{\int_\Sigma  |J^*(\sum_{i=1}^k a_i\varphi_i) |^2da}\\
&\leq \frac{1}{(n-p)c}\frac{\int_\Sigma  |d^\Sigma  \sum_{i=1}^k a_i\phi_i |^2da}{\int_\Sigma  | \sum_{i=1}^k a_i\phi_i |^2da}\\
&=\frac{1}{(n-p)c}\frac{\sum_{i=1}^k a_i^2\lambda_i^{[p]}}{\sum_{i=1}^k a_i^2}\\
&\leq \frac{1}{(n-p)c} \lambda_k^{[p]}.
\end{align*}
For a Euclidean ball with the radius $1/c$, it follows from Examples~\ref{exam1} and \ref{exam3} that $\sigma_k^{[p]}=(p+1)c$ and $\lambda_k^{[p]}=(p+1)(n-p)c^2$ for $1\leq p\leq n-1$ and $1\leq k\leq C_{n+1}^{p+1}$. So we get the equality and finish the proof.

\section{Final remarks}\label{sec6}

If we weaken the assumptions in Theorem~\ref{thm-bound}, then by using Raulot and Savo's un-weighted Reilly formula \eqref{eq-RS} we may obtain a non-sharp lower bound for the first non-zero eigenvalue of the operator $\mathcal{D}^{[p]}$.
\begin{thm}\label{thm-non-sharp}
Let $(M^{n+1},g)$ be an $(n+1)$-dimensional smooth compact connected oriented Riemannian manifold with boundary $\Sigma=\pt M$ and $1\leq p\leq n-1$ an integer. Assume that the $(p+1)$st Weitzenb\"{o}ck curvature of $M$ is non-negative, and the $(p+1)$-curvature $\sigma_{p+1}(\Sigma)$ and the $(n-p)$-curvature $\sigma_{n-p}(\Sigma)$ of $\Sigma$ in $M$ satisfy $\sigma_{p+1}(\Sigma)\geq (p+1)c$ and $\sigma_{n-p}(\Sigma)\geq (n-p)c$ for some constant $c>0$. Denote by $\sigma$ the first non-zero eigenvalue of the Dirichlet-to-Neumann operator $\mathcal{D}^{[p]}$. Then
\begin{equation*}
\sigma> \frac{(p+1)c}{2}.
\end{equation*}
\end{thm}
\begin{proof}
Let $\varphi\in \Omega^{p}(M)$ be an eigenform corresponding to the eigenvalue $\sigma$. Choosing $\omega=d\varphi\in \Omega^{p+1}(M)$ in the un-weighted Reilly formula \eqref{eq-RS} and noting as before $\delta d\varphi=0$, we deduce
\begin{align*}
0&=\int_M  |\nabla \omega  |^2+\langle W_{M}^{[p+1]}(\omega),\omega\rangle dv+2\int_\Sigma \langle \delta^\Sigma (J^*\omega),i_N \omega \rangle da+\int_\Sigma \mathcal{B}(\omega,\omega) da\\
&\geq 2\int_\Sigma \langle \delta^\Sigma (J^*(d\varphi)),i_N d\varphi \rangle da+\int_\Sigma \mathcal{B}(d\varphi,d\varphi) da\\
&=-2\sigma \int_\Sigma \langle \delta^\Sigma (J^*(d\varphi)),J^*\varphi \rangle da+\int_\Sigma \mathcal{B}(d\varphi,d\varphi) da\\
&=-2\sigma \int_\Sigma |J^*(d\varphi)|^2 da+\int_\Sigma \mathcal{B}(d\varphi,d\varphi) da.
\end{align*}
Next we use the equivalent expression for $\mathcal{B}(d\varphi,d\varphi)$ (see \cite[Thm.~3]{RS11}) to derive
\begin{align*}
\mathcal{B}(d\varphi,d\varphi)&=\langle  S^{[p+1]}(J^*d\varphi), J^*d\varphi\rangle+\langle  S^{[n-p]}(J^*(*d\varphi)), J^*(*d\varphi)\rangle\\
&\geq (p+1)c|J^*d\varphi|^2+(n-p)c|J^*(*d\varphi)|^2\\
&=(p+1)c|J^*d\varphi|^2+(n-p)c|i_N(d\varphi)|^2\\
&\geq (p+1)c|J^*d\varphi|^2,
\end{align*}
where we used the basic fact $|J^*(*\omega)|=|*^\Sigma(i_N\omega)|=|i_N\omega|$. Noting that $\int_\Sigma |J^*(d\varphi)|^2 da\neq 0$ (as $p\leq n-1$), we get immediately
\begin{align*}
\sigma \geq \frac{(p+1)c}{2}.
\end{align*}
However, $\sigma=(p+1)c/2$ cannot occur; otherwise from the argument above we will get
\begin{align*}
&\qquad \ \ \nabla d\varphi=0 \text{ on }M,\\
&0=i_N(d\varphi)=-\sigma J^*\varphi \text{ on }\Sigma.
\end{align*}
Then on the boundary $\Sigma$ we see
\begin{align*}
|d\varphi|^2=|i_N(d\varphi)|^2+|J^*(d\varphi)|^2=|i_N(d\varphi)|^2+|d^\Sigma (J^*\varphi)|^2=0,
\end{align*}
which implies $d\varphi=0$ on $M$, a contradiction. Therefore
\begin{equation*}
\sigma> \frac{(p+1)c}{2},
\end{equation*}
and the proof is complete.
\end{proof}

For the case of functions (i.e., $p=0$), the corresponding result as in Theorem~\ref{thm-non-sharp} is in \cite[Thm.~8]{Esc97}. In view of Escobar's conjecture (\cite[p.~115]{Esc99}), it would be an interesting open problem to see whether under the assumptions of Theorem~\ref{thm-non-sharp}, the sharp lower bound
\begin{align*}
\sigma \geq (p+1)c
\end{align*}
holds and further whether the rigidity result holds as well.
\begin{rem}
Under the assumptions of Theorem~\ref{thm-non-sharp} consider the case that $p\geq (n+1)/2$ and $M\subset \R^{n+1}$ is a Euclidean domain. Then $p\geq n-p$ and so the boundary curvatures satisfy $\sigma_p(\Sigma)\geq p\sigma_{n-p}(\Sigma)/(n-p)\geq pc$. Then by \cite[Thm.~1 (b)]{RS12} and \cite[Thm.~2.5]{Kar19} we get
\begin{align}
\sigma=\sigma_1^{[p]}\geq \nu_1^{[p]}\geq \frac{p+1}{p}\sigma_p(\Sigma)\geq (p+1)c.
\end{align}
Moreover, if $p>(n+1)/2$, then by \cite[Thm.~3]{RS12}, the equality $\sigma=(p+1)c$ holds if and only if $M$ is a ball of the radius $1/c$. So this case supports the above open problem.
\end{rem}

\noindent \textbf{Conflict of interest}: The author has no competing interests to declare that are relevant to the content of this article.

\noindent \textbf{Data availability statement}: Data sharing is not applicable to this article as no datasets were generated or
analyzed during the current study.


\bibliographystyle{Plain}

\end{document}